\documentclass[cmp]{svjour}  
\usepackage{amsmath}
\usepackage{amsfonts,amssymb}
\usepackage{color}
\definecolor{marin}{rgb}{0.,0.3,0.7}
\usepackage[colorlinks,citecolor=marin,linkcolor=marin,urlcolor=marin,bookmarksopen,bookmarksnumbered]{hyperref}
\usepackage{epsfig}
\usepackage{GGGraphics}

\journalname{}
\date{ \phantom{b} \vspace{45mm}\phantom{e}}

\def\iu{{\rm i}}
\def\e{{\rm e}}

\def\Z{{\mathbb Z}}

\def\real{{\mathbb R}}
\def\eps{\varepsilon}
\def\bigo{{\mathcal O}}

\def\signed #1 {{\unskip\nobreak\hfil\penalty50
    \hskip2em\hbox{}\nobreak\hfil\small#1
    \parfillskip=0pt \finalhyphendemerits=0 \par}}

\newcommand{\sfrac}[2]{\mbox{\footnotesize$\displaystyle\frac{#1}{#2}$}}
\newcommand\calM{{\cal M}}
\newcommand\calN{{\cal N}}

\newcommand\calE{{\cal E}}

\newcommand\calK{{\cal K}}

\newcommand\calU{{\cal U}}
\newcommand\calR{{\cal R}}

\newcommand\bfk{{\mathbf k}}

\newcommand\bfp{{\mathbf p}}
\newcommand\bfq{{\mathbf q}}

\newcommand\bfv{{\mathbf v}}
\newcommand\bfw{{\mathbf w}}
\newcommand\bfy{{\mathbf y}}
\newcommand\bfz{{\mathbf z}}

\newcommand\bfLambda{{\boldsymbol \Lambda}}
\newcommand\bfvartheta{{\boldsymbol \vartheta}}
\newcommand\bfvarpi{{\boldsymbol \varpi}}
\newcommand\bfomega{{\boldsymbol \omega}}

\newcommand\bfzero{{\mathbf 0}}

\newcommand\jvec{{\langle j \rangle}}

\begin{document}

\title{{Energy separation in oscillatory Hamiltonian systems
without any non-resonance condition}}
\titlerunning{Energy separation in oscillatory Hamiltonian systems}

\author{Ludwig Gauckler\inst{1} \and Ernst Hairer\inst{2} \and Christian Lubich\inst{3}}
\institute{Institut f\"ur Mathematik, TU Berlin, Stra\ss e des 17.~Juni 136, D-10623 Berlin, Germany. \email{gauckler@math.tu-berlin.de}
\and
Section de math\'ematiques, 2-4 rue du Li\`evre,
Universit\'e de Gen\`eve, CH-1211
Gen\`eve 4,
Switzerland. \email{Ernst.Hairer@unige.ch} 
\and
  Mathematisches Institut, Universit\"at T\"ubingen, Auf der Morgenstelle,
  D-72076 T\"ubingen, Germany. \email{Lubich@na.uni-tuebingen.de}}
\authorrunning{L. Gauckler, E. Hairer and Ch. Lubich}

\date{}

\maketitle
\begin{abstract}
We consider multiscale Hamiltonian systems in which harmonic oscillators with several high frequencies are coupled to a slow system. It is shown that the oscillatory energy is nearly preserved over  long times $\eps^{-N}$ for arbitrary $N>1$, where $\eps^{-1}$ is the size of the smallest high frequency. The result is uniform in the frequencies and does not require non-resonance conditions.
\end{abstract}

\section{Introduction and statement of the main result}

As has been discussed, e.g., in 
\cite{bambusi93eso,benettin87roh,benettin89roh,carati01tod,shudo06oos},  the problem of slow rates of thermalization in statistical mechanics can be understood  by studying multiscale problems of the following type:
For momenta $\bfp = (p_{0},p_{1},\ldots ,p_{n})$ and positions
$\bfq = (q_{0},q_{1},\ldots ,q_{n})$ with $p_{j},q_{j}\in\real^{d_{j}}$
consider a  Hamiltonian that couples high-frequency harmonic oscillators with a Hamiltonian of slow motion,
\begin{equation}\label{hamil}
H(\bfp ,\bfq ) =  H_{\bfomega}(\bfp ,\bfq ) + H_{\rm slow} (\bfp ,\bfq ) ,
\end{equation}
where the oscillatory and slow-motion energies are given by
\begin{equation}\label{energies}
H_{\bfomega}(\bfp ,\bfq )  =  \sum_{j=1}^n \sfrac 12\Bigl( |p_{j}|^2 +
 \omega_{j}^2 \,|q_{j}|^2\Bigr), \qquad
H_{\rm slow}(\bfp ,\bfq ) = \sfrac 12  |p_{0}|^2 + U(\bfq )
\end{equation}
with high angular frequencies
\begin{equation}\label{frequencies}
\omega_{j} \ge \frac1{\eps}, \quad
0<\eps\ll 1 .
\end{equation}
The coupling potential $U(\bfq )$ is assumed smooth with derivatives bounded independently of the small parameter $\eps$.

Weakly anharmonic systems similar to \eqref{hamil}--\eqref{frequencies} have been studied in numerous publications, from different viewpoints and with different mathematical techniques, ever since the seminal paper by Fermi, Pasta \& Ulam (\cite{fermi55son}, reprinted in \cite{gallavotti08tfp}). For some recent work, see  \cite{aoki06eti,dolgopyat11eti,gallavotti08tfp,hairer12ote,huveneers12tco}  and references therein.
The system \eqref{hamil}--\eqref{frequencies} has been considered in \cite{benettin89roh} with the double motivation of explaining slow rates of thermalization and the realization of holonomic constraints.

The equations of motion are
\begin{equation}\label{ode}
\ddot q_j + \omega_j^2 q_j =
 -\nabla_j U(\bfq ) , \qquad j=0,\dots,n,
\end{equation}
where $\nabla_j$ denotes the partial derivative with respect to $q_j$ and $\omega_{0}=0$.
We are interested in the energy exchange (or lack thereof) between the fast and the slow system, where in contrast to previous work we assume no condition on the frequencies other than the lower 
bound~(\ref{frequencies}).


\begin{theorem}\label{thm:main}
Fix an energy bound $E>0$, an integer $N\ge 1$, and further a radius $\rho>0$ and a set 
${K\subset \real^{d_0}}$ such that the potential $U$ has bounded derivatives of all orders in a $\rho$-neigh\-bour\-hood of $K\times 0\times\dots\times 0$.
Then there exist $C>0$ and
$\eps^* >0$ such that the following holds for $0<\eps \le \eps^*$: Whenever the frequencies satisfy $\omega_j\ge \eps^{-1}$ for $j=1,\dots,n$ and the
 initial values $(\bfp (0) , \bfq (0))$ are such that
 \begin{equation}\label{bound-energy}
H_{\bfomega}\bigl( \bfp (0) , \bfq (0) \bigr) \le E,
\end{equation}
then, along the solution of (\ref{ode}) to these initial values, the oscillatory energy deviates from its starting value by no more than
\begin{equation}\label{main-est}
\big| H_{\bfomega}\bigl( \bfp (t) , \bfq (t) \bigr) - H_{\bfomega}\bigl( \bfp (0) , \bfq (0) \bigr)\big|
\le C\, \eps ^{3/4}
\quad\ \hbox{for}\quad 0\le t \le \eps^{-N}  ,
\end{equation}
provided that $q_0(t)$ stays in the set $K$ for such long times.
The threshold~$\eps^*$ and the constant $C$ 
depend on $n$ and $N$, on the energy bound~$E$ and on bounds of derivatives of the
potential $U$.
\end{theorem}

We emphasize that no non-resonance condition is required in this result. The estimate is uniform in the frequencies $\omega_j\ge \eps^{-1}$. This sets Theorem~\ref{thm:main} apart from existing related results in Hamiltonian perturbation theory.

For   fixed strongly non-resonant frequencies as well as for exactly resonant frequencies that satisfy a diophantine non-resonance condition outside a resonance module, the long-time estimate (\ref{main-est}) is a  particular case of the energy separation results shown by Benettin, Galgani and Giorgilli  \cite{benettin89roh}. The proof of Theorem~\ref{thm:main} is based on modulated Fourier expansions with suitably modified frequencies and does not use the canonical transformation techniques of Hamiltonian perturbation theory on which the proofs in \cite{benettin89roh} are based. 
The idea of modifying frequencies to avoid almost-resonant situations has also been used in \cite{bambusi93eso} in a different context where all frequencies are sufficiently close to a single high frequency.

The results in \cite{benettin89roh} are proved over times that are exponentially long in negative fractional powers of $\eps$ in the case of a real-analytic potential. We expect that then also Theorem~\ref{thm:main} holds over exponentially long times, but we will not investigate this question in this paper. See, however, \cite{cohen03mfe} for a corresponding result over exponentially long times for a single high frequency proved via modulated Fourier expansions.

We remark that the order $\eps^{3/4}$ can be improved, with the same proof,  to $\eps^{1-\delta}$ for arbitrary $\delta>0$, with $\eps^*$ and $C$ depending additionally on $\delta$ and deteriorating as $\delta\to 0$.

The energy bound $H_\bfomega(\bfp,\bfq)\le {\rm Const.}$ implies $q_j=\bigo(\eps)$ for $j=1,\dots,n$. Theorem~\ref{thm:main} and conservation of the total energy therefore yield that also the energy $\frac12|p_0|^2+U(q_0,0,\dots,0)$ of the isolated slow system is nearly preserved over time $\eps^{-N}$ along solutions of (\ref{ode}).

We present a numerical example in the next section and prove Theorem~\ref{thm:main} in the remainder of the paper.

\section{Numerical illustration}
 
We consider the problem with $n=7$ frequencies
\[
\eps \omega_j = 1, 1+\eps^2, 1+\eps, 1+\eps^{3/4},1+\eps^{2/3}, 1+\eps^{1/2}, 2
\]
for $j=1,\ldots ,n$, and with the potential
\[
U(\bfq ) = \sfrac 12 \,q_{0}^2 + \bigl(a\, q_{0}+q_{1}+q_{2}+2\,q_{3}+
3\,q_{4}+q_{5}+q_{6}+3\,q_{7}\bigr)^3 
\]
depending on a parameter $a$.
The initial values are chosen as
\[
\begin{array}{rcl}
\bfq (0) &=& (1, 0.3\, \eps , 0.4\, \eps , 0.7\, \eps , -1.1\, \eps ,0.4\, \eps ,
-0.6\, \eps , -0.7\, \eps )\\[2mm]
\bfp (0) &=& ( -0.2 , 0.6 , 0.7 , -0.9 , -0.9 , 0.4 , -1.1 , 0.8 ).
\end{array}
\]
The numerical solution is computed with a trigonometric integrator (symplectic method of Deuflhard, see \cite[Chap.\,XIII]{hairer06gni}), which is applied with constant step size $\Delta t = 0.01\, \eps $. Taking the step size twice as large gives nearly identical figures.

\begin{figure}[h]
\centering
\global\def\path{#1}\input{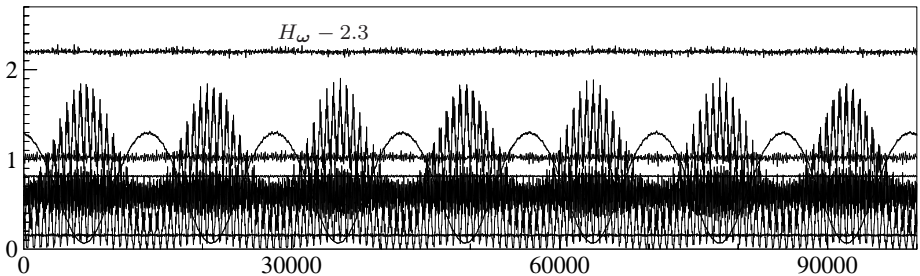}
\vspace{-4mm}
\caption{Individual oscillatory energies and shifted total oscillatory energy
with parameter $a=0.5$ and $\eps =0.005$.}
\label{fig:desgaut1}
\end{figure}

We consider two different values for the parameter $a$.
Figure~\ref{fig:desgaut1} shows the individual oscillatory energies
$E_{j}= \frac 12 (|p_{j}|^2 + \omega_{j}^2 |q_{j}|^2)$ as well as
the total oscillatory energy $H_{\bfomega}$ on an interval of
length $100\, 000$ for $a=0.5$ (for $a=1$ the energies explode)
and for $\eps=0.005$.
The initial oscillatory
energies are close to $0.22, 0.32, 0.65, 1.05, 0.17, 0.82, 1.3$. We observe that the oscillatory energies for the modes $j=4,5,6$ remain nearly constant, whereas for the other modes, which are closer to a 1:1 and 1:2 resonance, there is a significant energy exchange.  
In the figure
the constant value $2.3$ is subtracted from the total oscillatory energy, so that
it fits better into the picture. 
We repeated the experiment with other values of $\eps$ and found that
the maximal deviation (\ref{main-est})
of the total oscillatory
energy is
$4.56\cdot 10^{-1}$ for $\eps = 0.02$, 
$1.85\cdot 10^{-1}$ for $\eps = 0.01$, and
$9.54\cdot 10^{-2}$ for $\eps = 0.005$, which indicates an
$\bigo (\eps )$-behaviour.

\begin{figure}[h]
\centering
\global\def\path{#1}\input{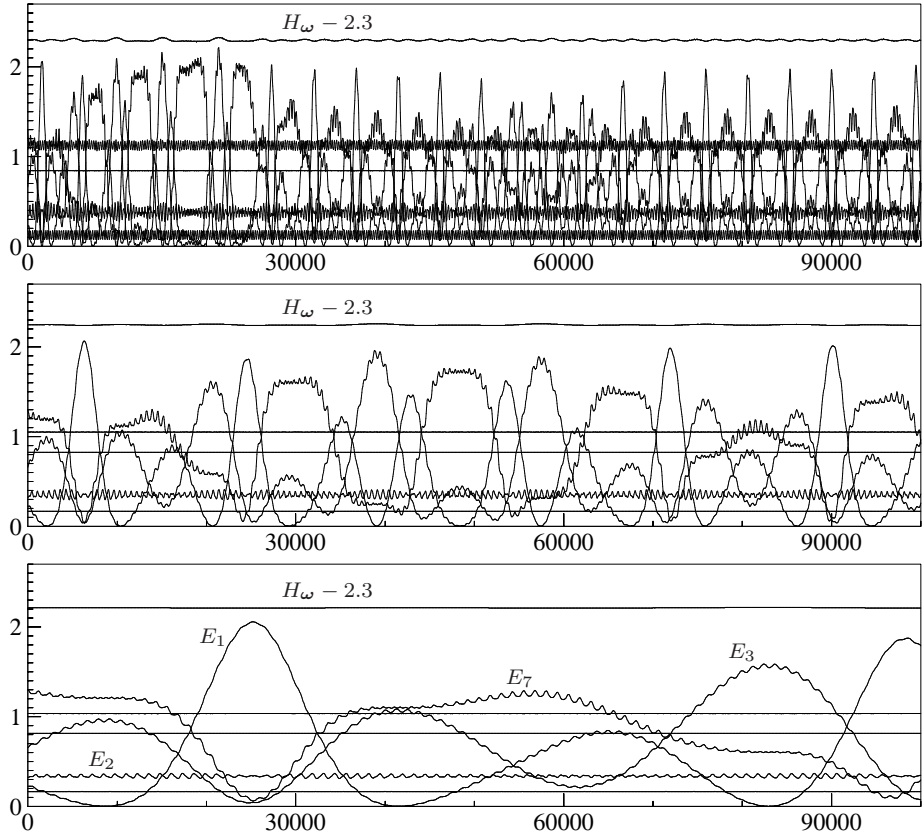}
\vspace{-4mm}
\caption{Individual oscillatory energies and total oscillatory energy
with parameter $a=\eps$.
The upper picture is for $\eps =0.02$,
the middle picture for $\eps =0.01$, and the lower picture for $\eps =0.005$.}
\label{fig:desgaut2}
\end{figure}

Figure~\ref{fig:desgaut2} shows the same experiment for the parameter $a=\eps$
and for three different values of $\eps$.
Here, the maximal deviation (\ref{main-est}) of the total oscillatory
energy is
$3.95\cdot 10^{-2}$ for $\eps = 0.02$, 
$1.41\cdot 10^{-2}$ for $\eps = 0.01$, and
$4.81\cdot 10^{-3}$ for $\eps = 0.005$.

\section{Modulated Fourier expansion with modified frequencies}

\subsection{Changing almost-resonant frequencies to exactly resonant frequencies} 

We deal with almost-resonant frequencies by introducing modified frequencies close to the original frequencies such that almost-resonances among the original frequencies become exact resonances among the modified frequencies. 
For a multi-index $\bfk =(k_{1},\ldots ,k_{n})\in \Z^n$, we denote
$\| \bfk \| = |k_{1}|+ \ldots + |k_{n}|$. The
vector of frequencies is $\bfomega = (\omega_{1},\ldots , \omega_{n})$, and we write
$\bfk\cdot\bfomega = \sum_{j=1}^n k_{j}\omega_{j}$. 

In order to decide whether a linear combination $\bfk\cdot\bfomega$ of frequencies is considered as almost-resonant or as non-resonant we use the fact that there is a \emph{gap} in the linear combinations of frequencies.
\begin{lemma}\label{lemma:gap}
There exist $\mu>0$ depending only on $n$ and $N$ and $\alpha\in[\mu,\frac14]$ depending on the frequencies such that there is no multi-index $\bfk\in \Z^n$ with $\|\bfk\| \le N+1$ and $\eps^{-\alpha} \le |\bfk\cdot\bfomega|\le \eps^{-\alpha-\mu}$.
\end{lemma}
\begin{proof}
Let us denote the number of all multi-indices $\bfk\in\Z^n$ with $\| \bfk \| \le N+1$ by $M$, which obviously depends only on $n$ and $N$. Consider the logarithms to base $\eps^{-1}$ of the $|\bfk\cdot\bfomega|$ with $\| \bfk \| \le N+1$, which gives a set $\{
a_1,...,a_M \}$. With $\mu=1/(4M+4)$, there is an interval of width $\mu$ in $[\mu,\frac14]$ that does not contain any of the $a_i$. If we denote this interval by $[\alpha,\alpha+\mu]$, then there is no $\bfk\in \Z^n$ with $\|\bfk\| \le N+1$ and $\eps^{-\alpha} \le |\bfk\cdot\bfomega|\le \eps^{-\alpha-\mu}$. \qed
\end{proof}
Let $\alpha$ and $\mu$ be as in the previous lemma. We consider multi-indices $\bfk\in \Z^n$ with $\|\bfk\| \le N+1$ and $|\bfk\cdot\bfomega|\ge \eps^{-\alpha-\mu}$ as non-resonant, whereas those with $|\bfk\cdot\bfomega|\le\eps^{-\alpha}$ are considered as almost-resonant and collected in the set 
$$
\calR = \{ \bfk\in \Z^n\,:\, |\bfk\cdot\bfomega|\le\eps^{-\alpha} , \|\bfk\| \le N+1 \}.
$$ 
This set generates the module
$\calM$ that contains all finite linear combinations of multi-indices in $\calR$ with integer coefficients. 
We select a maximum number $d\le n$ of $\real$-linearly independent vectors $\bfk^1,\ldots,\bfk^d\in\calR$. Their real-linear span contains the module $\calM$.
%

We modify the almost-resonant frequencies $\omega_j$ to exactly resonant frequencies
$$
\varpi_j = \omega_j + \vartheta_j, \qquad j=1,\ldots,n,
$$
by determining $\bfvartheta=(\vartheta_1,\ldots,\vartheta_n)$ as a solution of minimal norm of the (possibly underdetermined) system of linear equations
$$
\bfk^i \cdot \bfomega + \bfk^i\cdot\bfvartheta = 0, \qquad i=1,\ldots,d.
$$
We then have
\begin{equation}\label{gamma}
\| \bfvartheta \| \le \gamma  \eps^{-\alpha},
\end{equation}
where $\gamma $ depends only on $n$ and $N$, and
\begin{equation}\label{inM}
\bfk \cdot \bfvarpi = 0 \quad\hbox{ for all } \bfk\in \calM,
\end{equation}
i.e., almost-resonances in the original frequencies become exact resonances in the modified frequencies. 
We further have the lower bound
\begin{equation}\label{notinM}
| \bfk\cdot\bfvarpi | \ge \tfrac12\, \eps^{-\alpha-\mu}
\quad\hbox{ for }\, \bfk\notin \calM, \,\|\bfk\|\le N+1,
\end{equation}
provided that $0<\eps\le \eps^*$, where $\eps^*$ depends only on $n$ and $N$, i.e., non-resonant multi-indices for the original frequencies are still non-resonant for the modified frequencies. It will be convenient to let $\varpi_0=0$ and $\vartheta_0=0$.

\subsection{Short-time approximation by a modulated Fourier expansion}
We let $\calN$ be a set of representatives of the equivalence classes in $\Z^n/\calM$, which are chosen such that for each $\bfk\in\calN$, the norm of $\bfk$ is minimal in the equivalence class $[\bfk]=\bfk+\calM$, and with $\bfk\in\calN$, also $-\bfk\in\calN$. We denote
$\calK = {\{ \bfk \in \calN \,:\, \| \bfk\| \le N \}}$.

We make the approximation ansatz (\emph{modulated Fourier expansion}) 
\begin{equation}\label{mfe}
q_{j}(t) \approx \sum_{\bfk\in\calK}
z_{j}^\bfk (\eps^{-\alpha}t)\, \e^{\iu (\bfk\cdot\bfvarpi )t},
\qquad j=0,1,\ldots ,n.
\end{equation}
The modulation functions $z_j^\bfk(\tau)$ for $\tau=\eps^{-\alpha}t$ with $\alpha$ from Lemma~\ref{lemma:gap} are  determined such that all their derivatives with respect to $\tau$ are bounded independently of $\eps$ and that they yield a small defect $\delta_j^\bfk$ in the modulation equations,
which are derived by inserting the expansion (\ref{mfe}) into the differential equation (\ref{ode}) and comparing the coefficients of~$\e^{\iu (\bfk\cdot\bfvarpi )t}$:
\begin{eqnarray}
&&\hspace{-2mm}\bigl(\varpi_{j}^2 - (\bfk\cdot\bfvarpi )^2\bigr)\, z_{j}^\bfk
 + 2\,\iu (\bfk\cdot\bfvarpi )\, \eps^{-\alpha}\frac{dz_{j}^\bfk}{d\tau} 
+ \eps^{-2\alpha} \frac{d^2z_{j}^\bfk}{d\tau^2} 
\label{mod-eq} 
\\
&& \hspace{4cm}
 = (2\varpi_j\vartheta_j - \vartheta_j^2)z_j^\bfk - \nabla_{j}^{-\bfk}\, \calU (\bfz )
 +\delta_j^\bfk.
\nonumber
\end{eqnarray}
Here $ \nabla_{j}^{-\bfk}$ denotes the gradient with
respect to $z_{j}^{-\bfk}$ and the modulation potential $\calU(\bfz)$ for $\bfz=(z_j^\bfk)$ is given as
\begin{eqnarray*}
&&
\calU (\bfz ) = 
U(\bfz^\bfzero_0 ) 
+ \sum_{m=1}^N
\sum_{j_{1},\ldots ,j_{m}=0}^n
\sum_{\bfk^1+\ldots +\bfk^m \in\calM}\frac 1 {m!}\, \partial_{j_1}\ldots\partial_{j_m}U(\bfz^\bfzero_0 )
\bigl( z_{j_{1}}^{\bfk^1},\ldots ,z_{j_{m}}^{\bfk^m} \bigr),
\end{eqnarray*}
where $\bfz^\bfzero_0 = (z_{0}^\bfzero , 0,\ldots ,0)$ and the indices
$(j_{l},\bfk^l ) = (0,\bfzero )$ are excluded from the sum. Note that $z_j^\bfk$ are vectors in $\real^{d_j}$ and correspondingly $\partial_{j_1}\ldots\partial_{j_m}U(\bfz^\bfzero_0 )$ is a multilinear form on $\real^{d_{j_1}}\times\dots\times\real^{d_{j_m}}$; in contrast $\nabla_j U$ denotes the gradient, a column vector in $\real^{d_j}$. In the following we
let $\jvec=(0,\dots,1,\dots,0)$ denote the $j$th unit vector in~$\Z^n$.

\begin{proposition}\label{thm:mfe}
In the situation of Theorem~\ref{thm:main}, the solution $q_j(t)$ $(j=0,\dots,n)$ of (\ref{ode}) admits an expansion
\begin{equation}\label{mfe2}
q_{j}(t)= \sum_{\bfk\in\calK}
z_{j}^\bfk (\eps^{-\alpha}t)\, \e^{\iu (\bfk\cdot\bfvarpi )t} + r_j(t),
\qquad 0\le t \le \eps^\alpha,
\end{equation}
where the coefficient functions $z_j^\bfk(\tau)$ satisfy $z_j^{-\bfk}=\overline{z_j^\bfk}$ and are bounded for $0\le\tau\le 1$, together with all their derivatives with respect to $\tau$ up to any fixed order,  by
$$
|z_0^\bfzero (\tau)| \le C_1, \quad |z_j^{\pm\jvec}(\tau)| \le C_1\omega_j^{-1}\le C_1\eps,
$$
and for all other $(j,\bfk)$,
\[
 |z_j^\bfk(\tau)| \le C_1 |\varpi_j^2 - (\bfk\cdot\bfvarpi)^2|^{-1} \eps^{\|\bfk\|}.
\]
The functions $z_j^\bfk$ satisfy the modulation equations (\ref{mod-eq}) with a defect bounded by
\begin{equation}\label{eq:defect}
|\delta_j^\bfk(\tau)| \le C_2 \eps^{N+1}, \qquad 0\le\tau\le 1.
\end{equation}
The remainder term is bounded by
$$
\bigl(\omega_j^2 |r_j(t)|^2 + |\dot r_j(t)|^2\bigr)^{1/2} \le C_3\eps^{N+1}, \qquad 0\le t \le \eps^\alpha.
$$
The constants $C_1,C_2,C_3$ are independent of $\eps$ and the frequencies
$\omega_{j}\ge \eps^{-1}$, but depend on $n$ and $N$, on the energy bound $E$ and on bounds of derivatives of the
potential~$U$.
\end{proposition}

This will be proved in Section~\ref{sec:proof-mfe}. The proof follows the lines of previous proofs for modulated Fourier expansions; compare for example 
\cite[Chap.\,XIII]{hairer06gni} and, in technically more complicated situations than here, \cite{cohen08lta,gauckler12mes,hairer12ote}.

\subsection{The almost-invariant}

We introduce the functions
$y_{j}^\bfk (t) = z_{j}^\bfk (\eps^{-\alpha}t)\, \e^{\iu (\bfk\cdot\bfvarpi) t}$,
which we collect in the vector
$\bfy = (y_{j}^\bfk )$ for $ j=0,\ldots, n$ and $\bfk\in \calK $.
Since $\bfk\cdot\bfvarpi=0$ for $\bfk\in\calM$, the modulation potential satisfies
$
\calU (\bfy ) = \calU (\bfz ) .
$
By (\ref{mod-eq}),  we have with rotated defects $d_j^\bfk(t)=\delta_{j}^\bfk(\eps^{-\alpha}t)\e^{\iu (\bfk\cdot\bfvarpi) t}$ that
\begin{equation}\label{y-ode}
\ddot y_{j}^\bfk + \omega_{j}^2 \, y_{j}^\bfk =
 - \nabla_{j}^{-\bfk}\, \calU (\bfy ) + d_{j}^\bfk ,
\end{equation}
where $ \nabla_{j}^{-\bfk}$ denotes the partial derivative with
respect to $y_{j}^{-\bfk}$. 

Without the defects $d_j^\bfk$, equation \eqref{y-ode} is again a Hamiltonian system. Since the sum in $\calU (\bfy )$ is over multi-indices $\bfk^1,\dots,\bfk^m$ with $\bfk^1+\dots+\bfk^m\in\calM$, the extended potential $\calU$ is invariant under the action of the one-parameter group $S(\theta)\bfy=(\e^{\iu\theta(\bfk\cdot\bfvarpi)}y_j^\bfk)$: 
$$
\calU(S(\theta)\bfy) = \calU(\bfy)\quad\hbox{for all $\theta\in \real$,}
$$
because $\bfk\cdot\bfvarpi=0$ for $\bfk\in\calM$. By Noether's theorem this leads to a conserved quantity $\calE$ of the system \eqref{y-ode} considered without defects. As we now show, this quantity is almost conserved in our situation including the defects $d_j^\bfk$: We have 
$$
-\sum_{j=0}^n\sum_{\bfk\in \calK} \iu (\bfk\cdot\bfvarpi)\, y_j^{-\bfk} \,\nabla_j^{-\bfk}\calU(\bfy) =
\frac d{d \theta}\,\calU(S(\theta)\bfy)\Big\vert_{\theta=0} = 0,
$$
and for
\begin{equation}\label{calE}
\calE (\bfy ,\dot \bfy ) = -\iu \sum_{j=0}^n\sum_{\bfk\in \calK}
  (\bfk\cdot\bfvarpi ) \, y_{j}^{-\bfk} \, \dot y_{j}^{\bfk} 
\end{equation}
we thus obtain by multiplying (\ref{y-ode}) with $-\iu(\bfk\cdot\bfvarpi) y_j^{-\bfk}$ and summing over $j$ and $\bfk$ that
\[
\frac d{d t}\calE (\bfy (t),\dot \bfy (t)) = -\iu \sum_{j=0}^n\sum_{\bfk\in \calK}
  (\bfk\cdot\bfvarpi ) \,y_{j}^{-\bfk}(t)\, d_{j}^\bfk(t)  .
\]
Since the right-hand side is $\bigo (\eps^{N+1} )$ by the estimates of Proposition~\ref{thm:mfe},  $\calE$ is an almost-invariant. Moreover, $\calE$ turns out to be close to the oscillatory energy $H_{\bfomega}$ along the solution of (\ref{ode}).

\begin{proposition} \label{thm:inv}
In the situation of Proposition~\ref{thm:mfe} we have for $0\le t \le \eps^\alpha$
\begin{eqnarray*}
&&
|\calE (\bfy (t),\dot \bfy (t)) - \calE (\bfy (0),\dot \bfy (0)) | \le C\eps^{N+1}
\\
&&
|\calE (\bfy (t),\dot \bfy (t)) - H_{\bfomega} (\bfp(t),\bfq(t))| \le C' \eps^{3/4}.
\end{eqnarray*}
The constants $C$ and $C'$ are independent of $\eps$ and the frequencies
$\omega_{j}\ge \eps^{-1}$.
\end{proposition}

\begin{proof}
The first inequality has already been derived above. Concerning the second inequality, we note that with $\tau=\eps^{-\alpha} t$ we have for $\bfk\ne\bfzero$
$$
\dot y_j^\bfk(t) = 
\e^{\iu(\bfk\cdot\bfvarpi)t}\Bigl( \eps^{-\alpha}\frac{d z_j^\bfk}{d\tau}(\tau) +
\iu(\bfk\cdot\bfvarpi)z_j^\bfk(\tau) \Bigr) = \iu(\bfk\cdot\bfvarpi)y_j^\bfk( t) +
\bigo(\eps^{1-\alpha}).
$$ 
With the bounds of Proposition~\ref{thm:mfe} it follows that (omitting the argument $t$)
\begin{eqnarray*}
\calE (\bfy,\dot \bfy) &=& 2\sum_{j=1}^n \varpi_j^2 |y_j^\jvec|^2 + \bigo(\eps^{1-\alpha})
\\
&=& 2\sum_{j=1}^n \omega_j^2 |y_j^\jvec|^2 + \bigo(\eps^{1-\alpha}).
\end{eqnarray*}
On the other hand, with Proposition~\ref{thm:mfe} we also obtain 
\begin{eqnarray*}
H_{\bfomega} (\bfp,\bfq) &=& \tfrac12 \sum_{j=1}^n \Bigl(
|\dot y_j^\jvec +\dot y_j^{-\jvec}|^2 + \omega_j^2 |y_j^\jvec +y_j^{-\jvec}|^2 \Bigr)
+ \bigo(\eps)
\\
&=& 2 \sum_{j=1}^n \omega_j^2 |y_j^\jvec|^2 + \bigo(\eps^{1-\alpha}).
\end{eqnarray*}
Since $\alpha\le\tfrac14$, the result follows. \qed
\end{proof}

\subsection{Transition from one interval to the next}
Propositions~\ref{thm:mfe} and \ref{thm:inv} are only valid on a short time interval $0\le t\le \eps^\alpha$. We now consider a second short time interval $\eps^\alpha\le t\le 2\eps^\alpha$ and control in the following proposition the transition in the almost-invariant from the first to the second interval. 

\begin{proposition}
\label{thm:trans} Let the conditions of Proposition~\ref{thm:mfe} be fulfilled.
Let $z_j^\bfk(\eps^{-\alpha}t)$ for $0\le t \le \eps^\alpha$ be the coefficient functions as in Proposition~\ref{thm:mfe}
for initial data $(\bfp(0),\bfq(0))$,  and let $y_j^\bfk(t)=z_j^\bfk(\eps^{-\alpha}t)\e^{\iu(\bfk\cdot\bfvarpi)t}$ and
$\bfy(t)=(y^{\bfk}_j(t))$.
Let further $\widetilde\bfy(t)=({\widetilde y}^{\bfk}_j(t))$ be the corresponding functions of the modulated Fourier expansion for $0\le t \le \eps^\alpha$  to the initial data $(\bfp(\eps^\alpha),\bfq(\eps^\alpha))$, constructed as in Proposition~\ref{thm:mfe}.
Then,
$$
 | \calE(\bfy(\eps^\alpha),\dot\bfy(\eps^\alpha))-\calE(\widetilde{\bfy}(0),\dot{\widetilde\bfy}(0))|  \le C \eps^{N+1},
 $$
where $C$ is independent of $\eps$ and the frequencies
$\omega_{j}\ge \eps^{-1}$.
\end{proposition}
 
 The proof will be given in Section~\ref{sec:proof-mfe}.
 
\subsection{From short to long time intervals}

Now we put many short time intervals together to get the long-time result of Theorem~\ref{thm:main}. For $m=0,1,2,\dots$, let $\bfy_m(t)$ contain the summands of the modulated Fourier expansion starting from $(\bfp(m\eps^\alpha),\bfq(m\eps^\alpha))$. As long as $H_{\bfomega}(\bfp(m\eps^\alpha),\bfq(m\eps^\alpha))\le 2E$, Proposition~\ref{thm:inv} yields for $0\le t \le \eps^\alpha$
$$
|\calE (\bfy_m (t),\dot \bfy_m (t)) - \calE (\bfy_m (0),\dot \bfy_m (0)) | \le C\eps^{N+1}.
$$
By Proposition~\ref{thm:trans},
$$
 | \calE(\bfy_m(\eps^\alpha),\dot\bfy_m(\eps^\alpha))-\calE({\bfy}_{m+1}(0),\dot\bfy_{m+1}(0))|  \le C \eps^{N+1}.
 $$
 Summing up these estimates over $m$ and applying the triangle inequality yields, for $0\le \theta \le \eps^\alpha$,
$$
|\calE(\bfy_{m}(\theta),\dot\bfy_{m}(\theta)) -  \calE (\bfy_0 (0),\dot \bfy_0 (0)) | \le 2(m+1)C\eps^{N+1}.
$$
By Proposition~\ref{thm:inv}, we have at $t=m\eps^\alpha+\theta$
$$
|\calE (\bfy_m (\theta),\dot \bfy (\theta)) - H_{\bfomega} (\bfp(t),\bfq(t))| \le C' \eps^{3/4}.
$$
Combining these bounds at $t$ and at $0$ we obtain for $t\le\eps^{-N}$
$$
|H_{\bfomega} (\bfp(t),\bfq(t)) - H_{\bfomega} (\bfp(0),\bfq(0))| \le 2C t \eps^{N+1-\alpha} + 2C'\eps^{3/4}.
$$
This yields the bound of Theorem~\ref{thm:main}.

\section{Proof of Propositions~\ref{thm:mfe} and \ref{thm:trans}}
\label{sec:proof-mfe}

\subsection{Construction of the modulation functions}

We give an iterative construction of the functions $z^\bfk_j$ ($j=0,\dots,n$ and $\bfk\in\calK$)
such that after $M=\lceil (N+1)/\mu \rceil $ iteration steps, the defect in equations
(\ref{mod-eq}) 
is of size
$\bigo(\eps^{N+1})$. 
We denote by $\bfz^m=\bigl([z_j^\bfk]^m\bigr)$ the $m$th
iterate and 
distinguish between the following cases:
\begin{enumerate}
\item
For $j=0$ and $\bfk = \bfzero$ the first two terms in (\ref{mod-eq}) disappear and we iterate with a second order differential equation
for $[z_{0}^\bfzero]^{m+1}$:
$$
\eps^{-2\alpha} \frac{d^2[z_{0}^\bfzero]^{m+1}}{d\tau^2} 
 = 
- \nabla_{0}^{-\bfzero}\, \calU (\bfz^m ).
$$
\item
For $j\ne 0$ and $\bfk = \pm\jvec$ with
the $j$th unit vector $\jvec = (0,\ldots ,0,1,0,\ldots ,0)$, the first term in (\ref{mod-eq}) disappears and
we iterate using a first order differential equation
for~$[z_{j}^{\pm\jvec}]^{m+1}$:
\begin{eqnarray}
&&\hspace{-2mm}
  \pm 2\iu\, \varpi_j\, \eps^{-\alpha}\frac{d[z_{j}^{\pm\jvec}]^{m+1}}{d\tau} 
+ \eps^{-2\alpha} \frac{d^2[z_{j}^{\pm\jvec}]^m}{d\tau^2} 
\label{mod-it-jj} 
\\
&& \hspace{4cm}
 = 
 (2\varpi_j\vartheta_j - \vartheta_j^2)[z_j^{\pm\jvec}]^{m+1} - \nabla_{j}^{\mp\jvec}\, \calU (\bfz^m ).
\nonumber
\end{eqnarray}
We note that by (\ref{gamma}) we have for sufficiently small $\eps$
$$
\Bigl|\frac{2\varpi_j\vartheta_j - \vartheta_j^2}{2\varpi_{j} \eps^{-\alpha}}\Bigr| \le 2\gamma .
$$

\item
In all other cases we iterate with an explicit equation for~$[z_j^\bfk]^{m+1}$:
\begin{eqnarray}
&&\hspace{-2mm}
 \bigl(\varpi_{j}^2 - (\bfk\cdot\bfvarpi )^2\bigr)\, [z_{j}^\bfk]^{m+1}
 + 2\,\iu (\bfk\cdot\bfvarpi )\, \eps^{-\alpha}\frac{d[z_{j}^\bfk]^m}{d\tau} 
+ \eps^{-2\alpha} \frac{d^2[z_{j}^\bfk]^m}{d\tau^2} 
\label{mod-it-jk} 
\\
&& \hspace{4cm}
 = 
 (2\varpi_j\vartheta_j - \vartheta_j^2)[z_j^\bfk]^m - \nabla_{j}^{-\bfk}\, \calU (\bfz^m ).
\nonumber
\end{eqnarray}

 For $j\ne 0$ we note that
$\pm\jvec-\bfk\notin \calM$ for $\bfk\in\calK$ and $\bfk\ne\pm\jvec$, and therefore
we have by (\ref{notinM}) 
$$
\Bigl| \frac{(\bfk\cdot\bfvarpi)  \eps^{-\alpha}}{\varpi_{j}^2 - (\bfk\cdot\bfvarpi )^2}
\Bigr|
\le 2 \eps^{\mu}, \qquad
\Bigl| \frac{\eps^{-2\alpha}}{\varpi_{j}^2 - (\bfk\cdot\bfvarpi )^2}\Bigr| \le 2
\eps^{1-\alpha+\mu}
$$
and
$$
\Bigl| \frac{2\varpi_j\vartheta_j - \vartheta_j^2}{\varpi_{j}^2 - (\bfk\cdot\bfvarpi )^2} 
\Bigr|
\le 8\gamma \eps^\mu.
$$
For $j=0$ and $\bfk\notin\calM$ we have  similar estimates directly by (\ref{notinM}).
\end{enumerate}
We need initial values
$[z_0^\bfzero]^{m+1} (0), \frac{d}{d\tau} [z_0^\bfzero]^{m+1} (0)$
and $[z_j^{\pm \jvec}]^{m+1}(0) $ for $j\ne 0$. They are uniquely determined by the equations
(for $j=0,\dots,n$)
\begin{equation}\label{mod-init}
q_{j}(0) =  \sum_{\bfk\in\calK}  [z_{j}^\bfk]^{m+1} (0),\quad
\dot q_{j}(0) = \sum_{\bfk\in\calK} 
\Bigl( \eps^{-\alpha} \frac{d[z_{j}^\bfk]^{m+1}}{d\tau} (0)+ \iu (\bfk\cdot\bfvarpi ) \,
[z_{j}^\bfk]^{m+1} (0) \Bigr)
\end{equation}
after inserting \eqref{mod-it-jk} for $\bfk\ne\pm\jvec$ and replacing for $j\ne 0$ the derivative of $[z^{\pm\jvec}_j]^{m+1}$ with \eqref{mod-it-jj}.
The starting iterates are chosen as 
$[z_0^\bfzero]^0(\tau) $ equal to the solution at $t=\eps^\alpha \tau$ of $\ddot y = -\nabla_0 U(y,0,\dots,0)$ with $y(0)=q_0(0)$ and $\dot y(0)=\dot q_0(0)$,
and for all other $(j,\bfk)$ with $\bfk\ne\pm\jvec$, we set $[z_j^\bfk]^0(\tau)=0$. The  diagonal functions $[z_j^{\pm\jvec}]^0$ are then chosen as constant functions such that 
(\ref{mod-init}) is satisfied.

\subsection{Bounds of the modulation functions}

We estimate the modulation functions on intervals of length $\tau=\bigo(1)$, or equivalently $t=\bigo (\eps^\alpha)$.
With the above inequalities and the energy bound (\ref{bound-energy}) we obtain by induction on $m$, for $m=0,\dots,M$,
\[
[z_0^\bfzero]^m=\bigo(1), \quad [z_{j}^{\pm\jvec}]^m = \bigo (\omega_j^{-1} )
\]
and for all other $(j,\bfk)$,
$$
[z_j^\bfk]^m= \bigo\bigl( |\varpi_j^2-(\bfk\cdot\bfvarpi)^2|^{-1} \eps^{\|\bfk\|}  \bigr),
$$
and the same bounds hold for all their derivatives up to any prescribed order.

\subsection{Bounds of the defect}
We prove the estimate \eqref{eq:defect} of the defect $\delta_j^\bfk$ in Proposition~\ref{thm:mfe}. The defect $[\delta_j^\bfk]^m(\tau)$ in inserting the $m$th iterate $\bfz^m =([z_j^\bfk]^m)$ in the modulation equations (\ref{mod-eq}) satisfies
$$
[\delta_0^\bfzero]^m= -\eps^{-2\alpha}\biggl(  \frac{d^2[z_{0}^\bfzero]^{m+1}}{d\tau^2} - \frac{d^2[z_{0}^\bfzero]^{m}}{d\tau^2} \biggr)
$$
and for $j=1,\dots,n$
\begin{eqnarray*}
[\delta_j^{\pm\jvec}]^m &=& \mp2\iu \,\varpi_j\, \eps^{-\alpha}\biggl(
\frac{d[z_{j}^{\pm\jvec}]^{m+1}}{d\tau} - \frac{d[z_{j}^{\pm\jvec}]^{m}}{d\tau}\biggr)
\\
&&
+\ (2\varpi_j\vartheta_j-\vartheta_j^2)
\Bigl(
[z_{j}^{\pm\jvec}]^{m+1} - [z_{j}^{\pm\jvec}]^{m}\Bigr)
\end{eqnarray*}
and for all other $(j,\bfk)$
$$
[\delta_j^\bfk]^m =  -\bigl(\varpi_{j}^2 - (\bfk\cdot\bfvarpi )^2\bigr)\, 
\Bigl( [z_{j}^\bfk]^{m+1} - [z_{j}^\bfk]^{m} \Bigr).
$$
In the following we work with the $C^r([0,1])$ norm, defined for $\bfv=(v_j^\bfk)$ by
$$
\| \bfv \|_{C^r} = \max_{0\le \tau \le 1} \max_{0\le l \le r} \,\sum_{j=0}^n
\sum_{\bfk\in \calK} \Bigl| \frac{d^l}{d\tau^l} v_j^\bfk(\tau) \Bigr|,
$$
and use the notation
$$
\bfLambda\bfz = \bfv = (v_j^\bfk) \quad\hbox{with} \quad \left\{
\begin{array}{l}
v_0^\bfzero = \eps^{-2\alpha} z_0^\bfzero
\\[1mm]
v_j^{\pm\jvec} = 2\iu\,\varpi_j \eps^{-\alpha}  z_j^{\pm\jvec}, \quad\  j=1,\dots,n
\\[1mm]
v_j^\bfk = (\varpi_j^2-(\bfk\cdot\bfvarpi)^2) z_j^\bfk \quad\ \hbox{else.}
\end{array}
\right.
$$
With this notation and the above formulas for the defect we have 
\[
\bigl|[\delta_j^\bfk]^m\bigr| \le C \| \bfv^{m+1} - \bfv^m \|_{C^2}.
\]
We therefore study $\bfv^{m+1} - \bfv^m$ and show by induction on $m$ that 
for $m\le M$ and $r=2M-2m+2$,
\begin{equation}\label{eq:induction}
\| \bfv^{m+1} - \bfv^m \|_{C^r} = \bigo(\eps^{m\mu}),
\end{equation}
which implies the estimate \eqref{eq:defect} of the defect in Proposition~\ref{thm:mfe} if we use $M\ge (N+1)/\mu$ iterations for the construction of the modulation functions. Estimate \eqref{eq:induction} is shown by analysing the iterative construction of modulation functions, where the $(m+1)$th iterates are the dominant terms and the $m$th iterates the non-dominant ones. 

We first consider the diagonal elements $[v_j^{\pm\jvec}]^m$. From (\ref{mod-init}) we obtain after inserting (\ref{mod-it-jk}) for $\bfk\ne\pm\jvec$ that
$$
| [v_j^{\pm\jvec}]^{m+1}(0) - [v_j^{\pm\jvec}]^{m}(0) | \le C\eps^\mu \| \bfv^m - \bfv^{m-1} \|_{C^2}.
$$
Solving the differential equation (\ref{mod-it-jj}) by the variation-of-constants formula
then yields that
$$
| [v_j^{\pm\jvec}]^{m+1}(\tau) - [v_j^{\pm\jvec}]^{m}(\tau) | \le C\eps^\mu \| \bfv^m - \bfv^{m-1} \|_{C^2}, \quad 0\le\tau\le1 .
$$
Repeated differentiation in (\ref{mod-it-jj}) further shows that for $\ell\le 2M-2m+2$ 
and for $0\le\tau\le1$ 
$$
\Bigl| \frac{d^\ell}{d\tau^\ell}[ v_j^{\pm\jvec}]^{m+1}(\tau) - \frac{d^\ell}{d\tau^\ell}[ v_j^{\pm\jvec}]^{m}(\tau) \Bigr| \le C\eps^\mu \| \bfv^m - \bfv^{m-1} \|_{C^{\ell+2}}.
$$
Using $2\alpha\ge \mu$ we also bound, for $r=2M-2m+2$,
$$
\| [v_0^\bfzero]^{m+1} - [v_0^\bfzero]^{m} \|_{C^r} \le C\eps^\mu \| \bfv^m - \bfv^{m-1} \|_{C^{r+2}}.
$$
Similarly we have from the iteration formulas and the above inequalities that
$$
\| \bfv^{m+1} - \bfv^m \|_{C^r} \le C\eps^\mu \| \bfv^{m} - \bfv^{m-1} \|_{C^{r+2}} .
$$
This proves \eqref{eq:induction} and the estimate of the defect of Proposition~\ref{thm:mfe}. 

\subsection{Solution approximation}
We consider the $M$th iterates of the modulation functions (with $M= \lceil (N+1)/\mu \rceil$) and omit the superscript $M$  on the modulation functions and their defects. 
The truncated modulated Fourier expansion
$$
\widetilde q_{j}(t)= \sum_{\bfk\in\calK} z_{j}^\bfk (\eps^{-\alpha}t)\, \e^{\iu (\bfk\cdot\bfvarpi )t}
$$
satisfies the perturbed second order differential equation
$$
\ddot {\widetilde q_j} + \omega_j^2 \widetilde q_j = - \nabla_j U (\widetilde \bfq) + d_j,
\quad j=0,\dots,n,
$$
with
$$
d_j(t) = \sum_{\bfk\in \calK} \delta_j^\bfk(\eps^{-\alpha}t) \e^{\iu(\bfk\cdot\bfvarpi)t} - \sum_{\bfk\in\calN\setminus\calK} \nabla_j^{-\bfk}\calU(\bfz(\eps^{-\alpha}t))\e^{\iu(\bfk\cdot\bfvarpi)t} + \rho_j(t)
= \bigo(\eps^{N+1}),
$$

\noindent
where  $\rho(t)$
denotes the gradient of  the remainder term in the truncated Taylor expansion of the potential $U$ about $\bfz_0^\bfzero(\eps^{-\alpha}t)$: omitting the arguments $t$ and $\tau=\eps^{-\alpha}t$, and writing $\widetilde\bfq= \bfz^\bfzero_0 + \widehat\bfq$,
\[
\rho_j= \nabla_j U(\widetilde\bfq) - \nabla_j U(\bfz^\bfzero_0 ) 
- \sum_{m=1}^N
\sum_{j_{1},\ldots ,j_{m}=0}^n
 \frac 1 {m!}\, \partial_{j_1}\ldots\partial_{j_m}\nabla_j U(\bfz^\bfzero_0 )
\bigl( \widehat q_{j_{1}},\ldots ,\widehat q_{j_{m}} \bigr).
\]
After subtracting the differential equation (\ref{ode}), standard estimates using the variation of constants formula and the Gronwall inequality then yield the bound of $r_j(t)=\widetilde q_j(t)-q_j(t)$ for $0\le t \le \eps^\alpha$ as stated in Proposition~\ref{thm:mfe}.

\subsection{Proof of Proposition~\ref{thm:trans}} 
Let $\bfz^m = ([z^\bfk_j]^m)$ and $\widetilde{\bfz}^m = ([\widetilde{z}^\bfk_j]^m)$ be the $m$th iterates of the construction of the modulation functions $\bfz$ and $\widetilde{\bfz}$ for initial data $(\bfp(0),\bfq(0))$ and $(\bfp(\eps^\alpha),\bfq(\eps^\alpha))$, respectively. With the iteration for $\widetilde{\bfz}^m$ and the defect formula for $\bfz^M$ the same arguments as in the estimates of the defect yield for $\bfw^m(\tau) = \bfLambda\bfz^M(1+\tau) - \bfLambda\widetilde{\bfz}^m(\tau)$ that 
\begin{equation}\label{interface}
\| \bfw^{M} \|_{C^2} = \bigo(\eps^{N+1}).
\end{equation}
The only difference is that the estimate of $[w_j^{\pm\jvec}]^{m+1}(0)$ now reads
\[
|[w_j^{\pm\jvec}]^{m+1}(0)| \le C \eps^\mu \| \bfw^{m} \|_{C^2} + C\eps^{N+1},
\]
a bound that is obtained by inserting \eqref{mfe2} in the defining relation \eqref{mod-init} for $[\widetilde{z}_j^{\pm\jvec}]^{m+1}(0)$ and by using the estimate of the remainder term from Proposition~\ref{thm:mfe}. Similarly the estimate of $w_0^\bfzero$ has to be modified. The bound of Proposition~\ref{thm:trans} follows from \eqref{interface}.

\begin{acknowledgement} 
This work has been supported by Fonds National Suisse, Project No. 200020-126638.
\end{acknowledgement}

\bibliographystyle{amsplain}
\providecommand{\bysame}{\leavevmode\hbox to3em{\hrulefill}\thinspace}
\providecommand{\MR}{\relax\ifhmode\unskip\space\fi MR }
\providecommand{\MRhref}[2]{%
  \href{http://www.ams.org/mathscinet-getitem?mr=#1}{#2}
}
\providecommand{\href}[2]{#2}

\end{document}